\documentclass[11pt, reqno]{amsart}
\usepackage{amsaddr}
\usepackage{latexsym,amsfonts}
\usepackage{amssymb,amsthm,upref,amscd}
\usepackage[T1]{fontenc}
\usepackage{times}
\usepackage{amscd}
\usepackage{enumerate}
\usepackage{mathrsfs}

\usepackage{amsmath}
\usepackage{color}
\usepackage{soul}
\usepackage{indentfirst}
\usepackage{comment}
\PassOptionsToPackage{reqno}{amsmath}
\usepackage{amsmath,amsfonts,amsthm,amssymb,bbm}
\usepackage{graphicx,color,dsfont}
\usepackage{enumitem}
\usepackage{fourier}

\newtheorem{thm}{Theorem}[section]
\newtheorem{prop}{Proposition}[section]

\newtheorem{lem}{Lemma}[section]
\newtheorem{cor}{Corollary}[section]

\newtheorem{assumption}{Assumption}[section]

\newtheorem{rem}{Remark}[section]
\theoremstyle{notation}

\newcommand{\R}{\mathbb{R}}

\numberwithin{equation}{section}
\newcommand{\N}{\mathbb{N}}
\newcommand{\Z}{\mathbb{Z}}

\newcommand{\eps}{\epsilon}

\newcommand{\wto}{\rightharpoonup}

\makeatletter \@addtoreset{equation}{section} \makeatother

\usepackage[textwidth=138mm,
textheight=215mm,
left=30mm,
right=30mm,
top=25.4mm,
bottom=25.4mm,
headheight=2.17cm,
headsep=4mm,
footskip=12mm,
heightrounded]{geometry}

\usepackage[colorlinks=true,urlcolor=blue,
citecolor=black,linkcolor=black,linktocpage,pdfpagelabels,
bookmarksnumbered,bookmarksopen]{hyperref}
\newcounter{const}
\author[T. Gou, X. Shen]{Tianxiang Gou$^{1}$, Xiaoan Shen$^{2}$}
\address[T. Gou]
{$^{1}$ School of Mathematics and Statistics, Xi’an Jiaotong University, 710049, Xi’an, Shaanxi, China}
\email{tianxiang.gou@xjtu.edu.cn}

\address[X. Shen]
{$^{2}$ Department of Mathematics, Statistics, and Computer Science, M/C 249, University of Illinois at Chicago, 851 S. Morgan Street, Chicago, IL 60607, USA}
\email{xshen30@uic.edu}



\title[Nonlinear Bound States]{Nonlinear bound states with prescribed angular momentum in the mass supercritical regime}
\date{\today}
	
	\thanks{{\it Acknowledgments}. T. Gou was supported by the National Natural Science Foundation of China (No. 12101483 $\&$ 12471113). The authors warmly thank Profs. I. Nenciu and C. Sparber for helpful discussions. {The authors also warmly thank the knowledgeable referee for constructive comments and suggestions.}}

\begin{document}
	
	\begin{abstract}
		In this paper, we consider the existence, orbital stability/instability and regularity of bound state solutions to nonlinear Schr\"odinger equations with super-quadratic confinement in two and three spatial dimensions for the mass supercritical case. Such solutions, which are given by time-dependent rotations of a non-radially symmetric spatial profile, correspond to critical points of the underlying energy function restricted on the double constraints consisting of the mass and the angular momentum. The study exhibits new pictures for rotating Bose-Einstein condensates within the framework of Gross-Pitaevskii theory. It is proved that there exist two non-radially symmetric solutions, one of which is local minimizer and the other is mountain pass type critical point of the underlying energy function restricted on the constraints. Moreover, we derive conditions that guarantee that local minimizers are regular, the set of those is orbitally stable and mountain pass type solutions are strongly unstable. The results extend and complement the recent ones in \cite{NSS}, where the consideration is undertaken in the mass subcritical case.
	\end{abstract}

	\maketitle
	
	\thispagestyle{empty}
	
	\section{Introduction}
	
	In this paper, we are concerned with solutions to the following focusing nonlinear Schr\"odinger equation (NLS),
	\begin{align} \label{equs}
		\left\{
		\begin{aligned}
			&\textnormal{i} \partial_t u=H u - \lambda |u|^{2 \sigma} u, \\
			& u(0)=u_0,
		\end{aligned}
		\right.  
	\end{align}
	where $d=2$ or $d=3$, $\lambda >0$, $0<\sigma<\frac{2}{(d-2)^+}$ and
	$$
	H:=-\frac 12 \Delta + V(x).
	$$
	Equation \eqref{equs} arises in various fields of physics, such as laser propagation or Bose-Einstein Condensation, see for example \cite{PS, SS}. Here the potential $V$ is a smooth confining potential, which is assumed to grow super-quadratically at infinity to ensure a strong confinement. More precisely, we assume that $V$ satisfies the following assumption.
	
	\begin{assumption} \label{a1}
		The potential $V \in C^{\infty}(\R^d,\R)$ is assumed to be radially symmetric and confining, i.e. $V(x) \to +\infty$ as $|x| \to \infty$. Moreover, there exist $k >2$ and $R>0$ such that for $|x|>R$,
		$$
		c_{\alpha} \langle x \rangle^{k-|\alpha|} \leq |\partial^{\alpha} V(x)| \leq C_{\alpha} \langle x \rangle^{k-|\alpha|}, \quad \forall \,\, \alpha \in \N^d,
		$$
		where $\langle x \rangle:=(1+|x|^2)^{1/2}$ and $c_{\alpha}, C_{\alpha}>0$ are constants. In addition, we assume that 
		\begin{align} \label{v}
		     \langle \nabla V(x), x \rangle > 0, \quad x \in \R^d \backslash \{0\}.
		\end{align}
	\end{assumption}
	
	{A typical example for an admissible potential satisfying Assumption \eqref{a1} is given by $V(x)=|x|^k$ for $k>2$.}
	
	\begin{rem}
	The assumption that the potential $V$ grows super-quadratically at infinity is to guarantee the compactness of sequences in the context of the double constraints. This is distinct with the discussion under the single mass constraint, which will be interpreted in the forthcoming parts.
	\end{rem}
	
	For simplicity, we shall assume that $V(x) \geq 0$ for any $x \in \R^d$. Indeed, in view of Assumption \ref{a1}, the potential $V$ is bounded from below. Hence $V \geq 0$ can be achieved by a simple gauge transform. 
	
	Here we shall denote the corresponding energy space by $\mathcal{H}^1$, which is defined by the completion of $C^{\infty}_0(\R^d)$ under the norm
	$$
	\|u\|_{\mathcal{H}^1}:=\|u\| + \|u\|_2, \quad \|u\|:=\|\nabla u\|_2 + \|V^{1/2} u\|_2,
	$$
	where
	$$
	\|u\|_p:=\left(\int_{\R^d} |u|^p \,dx \right)^{\frac 1p}, \quad 1 \leq p<+\infty.
	$$
	
	Observe that $\mathcal{H}^1 \subset H^1(\R^d)$, where ${H}^1(\R^d)$ is the usual Sobolev space defined by the completion of $C^{\infty}_0(\R^d)$ under the norm
	$$
	\|u\|_{{H}^1}:=\|\nabla u\| _2+ \|u\|_2.
	$$
	Accordingly, by using the well-known Gagliardo-Nirenberg's inequality in $H^1(\R^d)$ (see \cite{Ga, Ni}), we have the following result.
	
	\begin{lem} Let $0<\sigma<\frac{2}{(d-2)_+}$. Then there exists $C=C(d,\sigma)>0$ such that, for any $u \in \mathcal{H}^1$, 
		\begin{align}\label{gn}
			\|u\|_{2 \sigma +2}^{2 \sigma +2} \leq C \|\nabla u\|_2^{d\sigma} \|u\|_2^{\sigma(2-d) +2}.
		\end{align}
	\end{lem}
	
	We also present the relevant compact embedding in $\mathcal{H}^1$, which plays an important role in the discussion of the compactness of sequences.
	
	\begin{lem} \label{cembedding} \cite[Lemma 3.1]{Zh}
		Let  $0 \leq \sigma<\frac{2}{(d-2)_+}$. Then the embedding $\mathcal{H}^1 \hookrightarrow L^{2\sigma +2}(\R^d)$ is compact.
	\end{lem}
	
	Next we shall give the local well-posedness of solutions to \eqref{equs} in $\mathcal{H}^1$, which can be derived by applying Strichartz estimates established in \cite{YZ}, see for example \cite[Theorem 3.4]{Ca}.
	
	\begin{lem} \label{lwp}
		Let $V$ satisfy Assumption \ref{a1},  $0<\sigma<\frac{k+2}{k(d-2)_+}$ and $u_0 \in \mathcal{H}^1$. Then there exist $T_{max}>0$ and a unique solution $u \in C((-T_{max}, T_{max}), \mathcal{H}^1)$ to \eqref{equs} conserving the mass and the total energy, that is
		$$
		M(u(t))=M(u_0), \quad E(u(t))=E(u_0), \quad t \in (-T_{max}, T_{max}),
		$$
		where
		$$
		M(u):=\int_{\R^d} |u|^2 \,dx, 
		$$
		$$
		E(u):=\int_{\R^d} \frac 12 |\nabla u|^2 + V(x) |u|^2 - \frac{\lambda}{\sigma+1} |u|^{2 \sigma +2} \,dx.
		$$
		Moreover, there holds the blow-up alternative, i.e. $T_{max}=+\infty$ or $\lim_{t \to T_{max}^-}\|\nabla u(t, \cdot)\|_2=+\infty$.
	\end{lem}
	
	\begin{rem}
		It is unknown if the local well-posedness of solutions to \eqref{equs} remains valid for any $0<\sigma<2$ for $d=3$. See \cite{Ca} and the references therein for more details. 
	\end{rem}
	
	Indeed, apart from the mass and the total energy, there exists another important physical quantity, which is the so-called mean angular momentum of $u$ around a given rotation axis in $\R^d$. To further clarify, we shall denote the coordinate of a point in $\R^3$ by $(x_1, x_2, z)$ and assume, without loss of generality, that the rotation axis is the $z$-axis. Then the mean angular momentum of $u$ is given by
	$$
	L(u):=\langle u, L_zu \rangle_{2},
	$$
	where $L_z$ is the third component of the quantum mechanical angular momentum operator $\mathbb{L}:=-\textnormal{i} x \wedge \nabla$, namely
	$$
	L_zu=-\textnormal{i} \left(x_1 \partial_{x_2}u-x_2 \partial_{x_1} u\right).
	$$
	In $\R^2$, we shall make use of the standard convention and denote the coordinate of a point in $\R^2$ by $x=(x_1, x_2)$. The operators $\mathbb{L}$ and $L_z$ remain as above.
	
	Let $u \in C((-T_{max}, T_{max}), \mathcal{H}^1)$ be the solution to \eqref{equs} with initial datum $u_0 \in \mathcal{H}^1$. By straightforward calculations, one can show that the time-evolution of $L(u)$ under the flow of \eqref{equs} enjoys (see \cite{AMS}) that
	$$
	L(u(t)) + \textnormal{i} \int_0^t \int_{\R^d} |u(\tau)|^2 L_zV \,dxd\tau=L(u_0).
	$$ 
	If $V$ is radially symmetric, then it holds that $L_z V \equiv 0$. In this case, the dynamics of \eqref{equs} also satisfies the angular momentum conservation law, that is
	$$
	L(u(t))=L(u_0), \quad t \in (-T_{max}, T_{max}).
	$$
	
	Let $\Omega \in \R$ be a given angular velocity and $\Omega L_z$ be the generator of time-dependent rotations around the $z$-axis, namely for any $f \in L^2(\R^d)$,
	$$
	e^{\textnormal{i} t \Omega L_z} f(x)=f(e^{-t \Theta} x),
	$$ 
	where $\Theta$ is the skew symmetric matrix given by
	$$
	\Theta:=
	\left(
	\begin{matrix}
		0 &\Omega \\
		-\Omega & 0
	\end{matrix}
	\right) \,\, \mbox{for} \,\, d=2, \quad
	\left(
	\begin{matrix}
		0 &\Omega & 0 \\
		-\Omega & 0 & 0 \\
		0 & 0 & 0
	\end{matrix}
	\right) \,\, \mbox{for} \,\, d=3.
	$$
	It is easy to see that $\left(e^{-\textnormal{i} t \Omega L_z}\right)_{t \in \R}$ is a family of unitary operators, 
	$$
	e^{-\textnormal{i} t \Omega L_z} : \mathcal{H}^1 \to \mathcal{H}^1.
	$$ 
	Let $u \in C((-T_{max}, T_{max}), \mathcal{H}^1)$ be the solution to \eqref{equs} with initial datum $u_0 \in \mathcal{H}^1$. Define
	$$
	v(t,x):=e^{\textnormal{i} t \Omega L_z} u(t,x)=u(t, e^{-t \Theta} x).
	$$
	This implies that $u(t,x)=v(t, e^{t \Theta} x)$. It is simple to compute (see \cite{AMS}) that
	$$
	\partial_t u(t, x)=\partial_t v(t, e^{t \Theta} x)+ \left(\Theta \cdot e^{t \Theta} x \right) \cdot \nabla v(t, e^{t \Theta} x), \quad \Delta u(t,x)=\Delta v(t, e^{t \Theta} x).
	$$
	Observe that if $V$ is radially symmetric, then $V(e^{-t \Theta} x)=V(x)$. Consequently, we find that $v$ satisfies the following nonlinear Schr\"odinger equation,
	\begin{align} \label{equt}
		\left\{
		\begin{aligned}
			&\textnormal{i} \partial_t v=H v-\Omega L_z v-\lambda |v|^{2\sigma} v, \\
			& v(0)=u_0.
		\end{aligned}
		\right.
	\end{align}
	In particular, it holds that $L(v(t))=L(u(t))=L(u_0)$ for any $t \in (-T_{max}, T_{max})$. 
	
	Equation \eqref{equt} is a nonlinear Schr\"odinger equation with angular momentum rotation, which arises in the mean-field description of rotating Bose-Einstein condensates, see for example \cite{AMS, ANS, Se}. It is obvious that time-periodic solution of the form
	$$ 
	v(t,x)=e^{-\textnormal{i} \omega t} \phi(x)
	$$
	solves the associated stationary equation with rotation,
	\begin{align} \label{equ}
		H \phi -\omega \phi -\Omega L_z \phi=\lambda |\phi|^{2 \sigma} \phi.
	\end{align}
	This is usually regarded as the Euler-Lagrange equation of the associated Gross-Pitaevskii energy functional $E_{\Omega}$ restricted on $S(m)$, where
	$$
	E_{\Omega}(u):=\int_{\R^d} \frac 12 |\nabla u|^2 + V(x) |u|^2 - \frac{\lambda}{\sigma+1} |u|^{2 \sigma +2} \,dx-\Omega L(u),
	$$
	\begin{align} \label{c0}
		S(m):=\left\{u \in  \mathcal{H}^1 :  M(u)=m>0 \right\}.
	\end{align}
	{If there exists a critical point $u \in S(m)$ of the energy functional $E_{\Omega}$ restricted on $S(m)$, then it is a solution to \eqref{equ} and \eqref{c0}, where $\omega \in \R$ is Lagrange multiplier associated to the mass constraint.}
	Observe that, by Assumption \ref{a1}, 
	\begin{align} \label{lu}
		\begin{split}
			L(u) \leq \| |x| u\|_2 \|\nabla u\|_2 \leq C_{\eps}\| |x| u\|_2^2 + \eps \|\nabla u\|^2_2 
			&=C_{\eps} \left(\| |x|u\|_{L^2(|x| \leq \widetilde{R} )}^2 + \| |x|u\|_{L^2(|x| > \widetilde{R} )}^2 \right) + \eps \|\nabla u\|_2^2 \\
			& \leq C_{\eps}\widetilde{R}^2 \|u\|_2^2 + \frac 14 \|V^{1/2} u\|_2^2 +  \eps \|\nabla u\|_2^2,
		\end{split}
	\end{align}
	where $\eps>0$, $C_{\eps}>0$ and $\widetilde{R} \geq R$ is a large constant such that $2C_{\eps}^{1/2}|x| \leq V(x)$ for $|x| > \widetilde{R}$.
	Thereby, we find that $E_{\Omega}$ is well-defined in $\mathcal{H}^1$ and it is of class $C^1$ in $\mathcal{H}^1$. When $0<\sigma<\frac 2 d$, by Gagliardo-Nirenberg’s inequality \eqref{gn} jointly with \eqref{lu}, one then knows that the functional $E_{\Omega}$ restricted on $S(m)$ is bounded from below. Therefore, the existence of solutions to \eqref{equ}-\eqref{c0} can be proved by minimizing the functional $E_{\Omega}$ subject to $S(m)$, i.e. considering the following global minimization problem,
	\begin{align} \label{gmin0}
		e_{\Omega}(m):=\inf_{u \in S(m)} E_{\Omega}(u).
	\end{align}
	The existence and qualitative properties of minimizers to \eqref{gmin0} have been revealed in \cite{Se, Se1}. We also refer the readers to {\cite{ANS, BC, BCPY, CL, CRY, GLY, LY}} to the pertain consideration.
	
	As recently pointed out in \cite{NSS} a drawback of the aforementioned study with only a single mass constraint is that one does not know how much vorticity any solution to \eqref{equ}-\eqref{c0} carries and that one cannot exclude the possibility that solutions to \eqref{equ}-\eqref{c0} are radially symmetric. In particular, the numerical value $l \in \R$ of the mean angular momentum
	$$
	L(u)=\langle u, L_z u \rangle_2 
	$$
	is unknown. However, for $\lambda>0$ sufficiently large, it turns out in \cite{Se, Se1} that symmetry breaking of minimizers to \eqref{gmin0} can occur, see also \cite{BC1, BDZ} for numerical simulations. In \cite{NSS}, the authors initially posed whether it is possible to guarantee the existence of stationary profile $u$ to \eqref{equ}, which exhibits a certain predetermined angular momentum $L(u)=l>0$, for any given mass $m>0$. This yields a different point of view to explore solutions to rotating Bose-Einstein condensates within the framework of Gross-Pitaevskii theory. As a consequence, it would be interesting to investigate solutions to \eqref{equ} under the double constraints,
	\begin{align} \label{c}
		\int_{\R^d} |u|^2 \,dx=m, \quad \int_{\R^d} L_z u \overline{u}\,dx=l,
	\end{align}  
	where $m,l>0$ are a priori given. In this situation, solutions to \eqref{equ} and \eqref{c} correspond to critical points of the underlying energy functional $E$ restricted on $S(m,l)$, where
	$$
	S(m,l):=\left\{ u\in \mathcal{H}^1 : G_1(u)=0, G_2(u)=0\right\},
	$$
	$$
	G_1(u):=\int_{\R^d} |u|^2 \,dx-m, \quad G_2(u):=\int_{\R^d} L_z u \overline{u}\,dx-l.
	$$
	It turns out in \cite[Lemma 2.2]{NSS} that $S(m,l)$ is non-empty. It is clear that solutions to \eqref{equ} and \eqref{c} are not radially symmetric, because of $l>0$. Here $\omega, \Omega \in \R$ are unknown, which appear as Lagrange multipliers associated to the double constraints.

	When $0<\sigma< \frac 2 d$, by Gagliardo-Nirenberg’s inequality \eqref{gn}, one finds that that the energy functional $E$ restricted on $S(m,l)$ is bounded from below. In this case, one can introduce the following global minimization problem to study solutions to \eqref{equ} and \eqref{c},
	\begin{align} \label{gmin1}
		e(m,l):=\inf_{u \in S(m,l)} E(u).
	\end{align}
	With the help of \eqref{gmin1}, the existence and orbital stability of solutions to \eqref{equ} and \eqref{c} have been recently established in \cite{NSS}, where solutions actually correspond to such global minimizers to \eqref{gmin1}. As a complement of the study conducted in \cite{NSS}, we shall in Section \ref{regularity} derive conditions which guarantee that minimizers to \eqref{gmin1} are regular, see Proposition \ref{thm4}. While $ \sigma>\frac 2 d$, then $E$ restricted on $S(m,l)$ is unbounded from below. To see this, we define $u_{s}=s^{{d}/{2}} u(s\cdot)$ for $s>0$. Then $\|u_s\|_2=\|u\|_2$, $\langle L_zu_s, u_s \rangle_2= \langle L_z u, u \rangle_2$ and 
	\begin{align*}
		E(u_s)=\int_{\R^d} \frac {s^2}{2} |\nabla u|^2 + V\left(\frac{x}{s}\right) |u|^2 - \frac{\lambda s^{\sigma d}}{\sigma+1} |u|^{2 \sigma +2} \,dx.
	\end{align*}
	Since $c_0(1+|x|^2)^{\frac k2} \leq V(x) \leq C_0(1+|x|^2)^{\frac k2}$ for $|x|>R$ by Assumption \ref{a1}, then $E(u_s) \to -\infty$ as $s \to \infty$. Therefore, it is impossible to rely on the global minimization problem \eqref{gmin1} to investigate solutions to \eqref{equ} and \eqref{c}. 
	{Here solutions to \eqref{equ} and \eqref{c} indeed correspond to local minimizers or mountain pass type critical points of $E$ restricted on $S(m,l)$.
	}
	More precisely, when $\sigma>\frac{2}{d}$, as an extension of the results derived in \cite{NSS}, we have the following result.
	
	\begin{thm} \label{existence}
		Let $V$ satisfy Assumptions \ref{a1} and $\frac 2d<\sigma<\frac{2}{(d-2)_+}$. Then there exists $m_0>0$ such that, for any $0<m<m_0$, \eqref{equ} and \eqref{c} has two solutions $u_{m,l},v_{m,l} \in S(m,l)$ with $0<{E(u_{m,l})<E(v_{m,l})}$, where ${u_{m,l}}$ is a local minimizer of $E$ and ${v_{m,l}}$ is a mountain pass type critical point of $E$ restricted on $S(m,l)$.
	\end{thm}
	
	{
	\begin{rem}
	In Theorem \ref{existence}, the smallness of $m$ is to guarantee the existence of mountain pass structure for the energy functional $E$ restricted on $S(m, l)$, see Lemma \ref{mps0}, which is essential to establish the existence of solutions. It is unknown whether Theorem \ref{existence} remains valid for $m \geq m_0$. In view of \eqref{below}, one finds that $m_0$ depends explicitly on $r$, $\sigma$ and $d$.
	\end{rem}
	}
	
	{Define $B(r):=\left\{ u \in S(m, l) : \|u\|^2 \leq r \right\}$ and $\|u\|:=\|\nabla u\|_2 + \|V^{1/2} u\|_2$.} To prove Theorem \ref{existence}, we first analyze the structure of $E$ restricted on $B(r)$, see Lemma \ref{mps0}, by which we are able to introduce the following local minimization problem,
	\begin{align} \label{lmin1}
		e(m,l)=\inf_{u \in B(r)} E(u).
	\end{align}
	Next, utilizing Lemma \ref{cembedding} and the elements proposed in the proof of \cite[Proposition 2.4]{NSS}, we are able to show the compactness of any minimizing sequence to \eqref{lmin1}. Here we need that the potential $V$ grows super-quadratically at infinity. This then leads to the existence of solutions to \eqref{equ} and \eqref{c} in the spirit of Lemma \ref{mps0}, Corollary \ref{thm5} and Lagrange’s theorem. It is worth mentioning that the verification of the compactness of the sequence becomes more involved, because of the presence of the angular momentum constraint.
	
	To seek for the second solution to \eqref{equ} and \eqref{c}, we first establish the mountain pass structure of $E$ restricted on $S(m,l)$, see Lemma \ref{mps}. Note that it is hard to show that any Palais-Smale sequence for $E$ restricted on $S(m,l)$ at the mountain pass level $\gamma(m,l)$ is bounded in $\mathcal{H}^1$. Next, the essential argument lies in finding a Palais-Smale sequence for $E$ restricted on $S(m,l)$ at the mountain pass level $\gamma(m,l)$ around the related Pohozaev manifold $\mathcal{P}(m,l)$, see Lemma \ref{ps}, where
	$$
	\mathcal{P}(m,l):=\left\{ u \in S(m, l) : P(u)=0\right\}.
	$$ 
	Here $P(u)=0$ is the so-called Pohozaev identity to \eqref{equ} subject to $S(m,l)$. This is mainly inspired by the approach in \cite{Jeanjean}. Finally, by applying Lagrange’s theorem and verifying the compactness of the Palais-Smale sequence in $\mathcal{H}^1$, we arrive at the desired conclusion.  
	
	As a by-product, we also derive orbital stability of the set of minimizers to \eqref{lmin1}. 
	The result can be stated as follows.
	
	\begin{thm} \label{stable}
		Let $V$ satisfy Assumptions \ref{a1}, $\frac 2d<\sigma<\frac{k+2}{k(d-2)_+}$ and $\mathcal{B}(m,l)$ be the set of local minimizers obtained in Theorem \ref{existence}. Then $\mathcal{B}(m,l)$ is orbitally stable, this is, for any $\eps>0$, there exists $\delta>0$ such that if 
		$$
		\inf_{u \in \mathcal{B}(m,l)} \|u-u_0\|_{\mathcal{H}^1}<\delta,
		$$
		then the solution $v \in C(\R, \mathcal{H}^1)$ to \eqref{equt} with initial datum $u_0 \in \mathcal{H}^1$ satisfies
		$$
		\sup_{t \in \R} \inf_{u \in \mathcal{B}(m,l)} \|u-v(t, \cdot)\|_{\mathcal{H}^1}<\eps.
		$$
	\end{thm}
	
	In the mass supercritical case, one finds that the conservation laws, Gagliardo-Nirenberg’s inequality \eqref{gn} and \eqref{lu} cannot lead to the global existence of solutions to \eqref{equt}. This is different from the mass subcritical case. Here Lemma \ref{mps0} plays an important role to guarantee the global existence of solutions to \eqref{equt} with initial data near the set of minimizers to \eqref{lmin1}
	
	
	We now turn to investigate orbital instability of the mountain pass type solutions to \eqref{equ} and \eqref{c} obtained in Theorem \ref{existence}. For this, we shall addtionally assume that 
	\begin{align} \label{v1}
		\sum_{i, j=1}^d\partial_{ij}V(x)x_ix_j > 0, \quad x \in \R^{d} \backslash \{0\}. 
	\end{align}
	Moreover, we shall assume that, for any $x \in \R^d$, the function
	\begin{align} \label{v2}
		s \mapsto 2\langle \nabla V\left(\frac{x}{s}\right), x \rangle +\frac{1}{s} \sum_{i, j=1}^d\partial_{ij}V\left(\frac{x}{s}\right)x_ix_j 
	\end{align}
	is nonincreasing with respect to $s$ on $(0, +\infty)$. {A typical example for an admissible potential satisfying \eqref{v1} and \eqref{v2} is given by $V(x)=|x|^k$ for $k>2$}. Here the assumptions \eqref{v1} and \eqref{v2} are used principally to establish an alternative variational characterization of the mountain pass energy level $\gamma(m,l)$, see Lemma \ref{vc}. This along with the evolution of the virial quantity $\mathcal{M}[u(t)]$ (see Lemma \ref{virial}) then leads to the following result, where
	$$
	\mathcal{M}[u]:=\int_{\R^d} |x|^2 |u|^2 \,dx.
	$$
	
	\begin{thm} \label{unstable}
		Let $V$ satisfy Assumptions \ref{a1}, \eqref{v1} and \eqref{v2} hold, $\frac 2d<\sigma<\frac{k+2}{k(d-2)_+}$ and $v_{m,l} \in S(m,l)$ be the mountain pass type solution to \eqref{equ} obtained in Theorem \ref{existence}. Then $v_{m,l}$ is strongly unstable, that is, for any $\eps>0$, there exists $u_0 \in \mathcal{H}^1$ such that $\|v_{m,l}-u_0\|_{\mathcal{H}^1} <\eps$ and the solution to \eqref{equt} with initial datum $u_0$ blows up in finite time.
	\end{thm}
	
	The paper is organized as follow. In Section \ref{exist}, we shall prove the existence of solutions to \eqref{equ} and \eqref{c} and present the proof of Theorem \ref{existence}. Section \ref{dynamics} is devoted to the discussion of orbital stability/instability of solutions to \eqref{equ} and \eqref{c} and contains the proofs of Theorems \ref{stable} and \ref{unstable}. Finally, in Section \ref{regularity}, we shall discuss regularity of minimizers to \eqref{gmin1} and \eqref{lmin1}, respectively.
    
    \section{Existence of Solutions} \label{exist}
	
	In this section, we are going to prove the existence of solutions to \eqref{equ} and \eqref{c} and establish Theorem \ref{existence}. For this, we first introduce the corresponding Pohozaev functional related to \eqref{equ} defined by
	$$
	P(u):=\int_{\R^d}|\nabla u |^2 - \langle \nabla V(x), x \rangle |u|^2 - \frac{\lambda \sigma d}{\sigma+1} |u|^{2 \sigma +2} \,dx.
	$$
	Indeed, if $u \in \mathcal{H}^1$ is a solution to \eqref{equ}, then $P(u)=0$.
    \begin{rem}
    The Pohozaev identity $P(u)=0$ can be proved by multiplying \eqref{equ} by $x\cdot \nabla\overline{u}$ and $\overline{u}$ respectively and integrating over $\R^{d}$. It is important to note that the rotation term does not appear in $P(u)$.
    \end{rem}
    Throughout this section, we shall always assume that the assumptions of Theorem \ref{existence} hold. Moreover, we denote by $u_s$ a scaling of $u \in \mathcal{H}^1$ defined by
	\begin{equation}\label{scaling}
	    	u_s=s^{\frac d2}u(s \cdot), \quad s>0.
	\end{equation}
	To begin with, we shall derive the existence of local minimizers.
	
	
	\begin{lem} \label{mps0}
		For any $r>0$, there exists $m_0>0$ such that, for any $0<m<m_0$,
		$$
		\inf_{u \in B(r/16)}E(u)<\inf_{u \in B(r) \backslash B(r/2)}E(u),
		$$
		where $B(r):=\left\{ u \in S(m, l) :  \|u\|^2 \leq  r \right\}$ and $\|u\|:=\|\nabla u\|_2 + \|V^{1/2} u\|_2$.
	\end{lem}
	\begin{proof}
		Using Gagliardo-Nirenberg's inequality \eqref{gn}, we find that, for any $u \in S(m,l)$,
		\begin{align} \label{below}
			\begin{split}
				E(u) &\geq \int_{\R^d} \frac 12 |\nabla u |^2  + V(x) |u|^2 \,dx - C \left(\int_{\R^d} |\nabla u|^2 \,dx\right)^{\frac{d\sigma}{2}} m^{\frac{\sigma}{2}(2-d) +1} \\
				&\geq\left(\frac 14 -C \|u\|^{d\sigma-2} m^{\frac{\sigma}{2}(2-d) +1} \right) \|u\|^2.
			\end{split}
		\end{align}
		Then, for any $r>0$, there exists $m_0>0$ such that, for any $0<m<m_0$,
		\begin{align} \label{below1}
			\inf_{u \in B(r) \backslash B(r/2)}E(u) > \frac {r}{16}.
		\end{align}
		Observe that
		$$
		E(u) \leq \int_{\R^d} \frac 12 |\nabla u |^2  + V(x) |u|^2 \,dx.
		$$
		Therefore, we have that
		\begin{align} \label{below2}
			\inf_{u \in B(r/16)}E(u) \leq \frac {r}{16}.
		\end{align}
		Hence the proof is complete.
	\end{proof}
	
	\begin{thm} \label{thm1}
		For any $r>0$, there exists $m_0>0$ such that, for any $0<m<m_0$, \eqref{equ} and \eqref{c} has a solution $u_{m ,l} \in B(r)\subset S(m, l)$ with $E(u_{m,l})>0$. 
	\end{thm}
	\begin{proof}
		Define 
		\begin{align} \label{lmin}
			e(m, l)=\inf_{u \in B(r)} E(u).
		\end{align}
		It is easy to see that $e(m,l)>-\infty$, {because of \eqref{below}.}
		 Let $\{u_n\} \subset B(r)$ be a minimizing sequence to \eqref{lmin}. Observe that $\{u_n\}$ is bounded in $\mathcal{H}^1$. Then there exists $u_{m,l} \in \mathcal{H}^1$ such that $u_n \wto u_{m,l}$ in $\mathcal{H}^1$ as $n \to \infty$. By Lemma \ref{cembedding}, we know that $u_n \to u_{m,l}$ in $L^{2+2\sigma}(\R^d)$ as $n \to \infty$ for any $0 \leq \sigma<\frac{2}{(d-2)^+}$. It follows that $G_1(u_{m,l})=0$ and 
		$$
		E(u_{m,l}) \leq \liminf_{n \to \infty} E(u_n).
		$$
		At this point, to achieve that $u_{m,l}$ is a minimizer to \eqref{lmin}, we only need to show that $G_2(u_{m,l})=0$. This can be done analogously to the steps proposed in the proof of \cite[Proposition 2.4]{NSS}.
		{It follows from Lemma \ref{mps0} that $u_{m,l} \in S(m,l)$ is an interior minimizer to \eqref{lmin}. Then, using Corollary \ref{thm5}, we have that it is a solution to \eqref{equ} and \eqref{c}. }
		Next we show that $e(m,l)>0$. Since $u_{m,l}$ is a solution to \eqref{equ}, then $P(u_{m,l})=0$. Observe that
		\begin{align*}
			E(u_{m,l})&=E(u_{m,l})-\frac{1}{\sigma d} P(u_{m,l})\\
			&=\left(\frac 12 -\frac{1}{\sigma d}\right) \int_{\R^d} |\nabla u_{m,l}|^2 \,dx + \int_{\R^d} V(x) |u_{m,l}|^2 \,dx + \frac{1}{\sigma d} \int_{\R^d} \langle \nabla V(x), x \rangle |u_{m,l}|^2 \,dx,
		\end{align*}
		which implies that $E(u_{m,l})>0$. This completes the proof.
	\end{proof}
	
	Let $r>0$ be given and $m_0>0$ be the constant determined in Lemma \ref{mps0}. In the following, we are going to establish the existence of mountain pass type solutions to \eqref{equ} and \eqref{c}. To this end, we first introduce the mountain pass energy level.
	
	\begin{lem} \label{mps}
		Define
		$$
		\gamma(m, l):=\inf_{g \in \Gamma(m ,l)} \max_{0 \leq \tau \leq 1} E(g(\tau)),
		$$
		$$
		\Gamma(m,l):=\left\{ g \in C([0, 1], S(m ,l)) : g(0)=u_{m, l}, g(1)=(u_{m, l})_{s_0}\right\},
		$$
		where $u_{m,l} \in B(r)\subset S(m,l)$ is the solution obtained in Theorem \ref{thm1}, ${(u_{m,l})_{s_0}}$ is defined by \eqref{scaling} and $s_0>0$ is a large constant such that $E((u_{m,l})_{s_0})<0$. Then it holds that
		$$
		\gamma(m, l) >\max\left\{E(u_{m, l}), E((u_{m,l})_{s_0})\right\}>0.
		$$
	\end{lem}
	\begin{proof}
		Define
		$$
		g(\tau):=\left(1-\tau+ \tau s_0\right)^{\frac d 2}u_{m,l}((1-\tau)x + \tau s_0 x), \quad 0 \leq \tau \leq 1.
		$$
		Clearly, we have that $ g \in \Gamma(m,l)$. This shows that $\Gamma(m, l) \neq \emptyset$. Note that for any $g \in \Gamma(m,l)$, $g(0)=u_{m,l}$ and $g(1)=(u_{m,l})_{s_0}$. In addition, from Lemma \ref{mps0} and Theorem \ref{thm1}, we know that $\|u_{m,l}\|^2 < r/2$ and $\|(u_{m,l})_{s_0}\|^2>r$. As a consequence, there exists $0<\tau_0<1$ such that $\|g(\tau_0)\|^2=3r/4$. It then follows from Lemma \ref{mps0} that
		\begin{align*}
			\max_{0 \leq \tau \leq 1} E(g(t)) \geq E(g(\tau_0)) \geq \inf_{u \in B(r) \backslash B(r/2)}E(u) >\inf_{u \in B(r/16)}E(u) \geq \inf_{u \in B(r)} E(u)=E(u_{m,l})>0.
		\end{align*}
		This completes the proof.
	\end{proof}
	
	To find a Palais-Smale sequence of $E$ restricted on $S(m,l)$ at the level $\gamma(m,l)$ around the associated Pohozaev manifold, we shall introduce the following auxiliary functional $F: \mathcal{H}^1 \times \R \to \R$ as
	$$
	F(u, s):=E(s \ast u)=\int_{\R^d} \frac {e^{2s}}{2} |\nabla u|^2 + V\left(\frac{x}{e^s}\right) |u|^2 - \frac{\lambda e^{\sigma ds}}{\sigma+1} |u|^{2 \sigma +2} \,dx,
	$$
	where $s \ast u:=e^{ds/2} u(e^s \cdot)$. This is inspired by the ideas in \cite{Jeanjean}.
	
	\begin{lem} \label{mps1}
		Define
		$$
		\widetilde{\gamma}(m,l):=\inf_{\widetilde{g} \in \widetilde{\Gamma}(m,l)} \max_{0 \leq \tau \leq 1} F(\widetilde{g}(\tau)),
		$$
		$$
		\widetilde{\Gamma}(m,l):=\left\{\widetilde{g} \in C([0, 1], S(m, l) \times \R) : \widetilde{g}(0)=(u_{m,l}, 0), \widetilde{g}(1)=((u_{m,l})_{s_0}, 0) \right\}.
		$$
		Then
		it holds that $\widetilde{\gamma}(m,l)=\gamma(m,l)$.
	\end{lem}
	\begin{proof}
		Note that $\Gamma(m,l) \times \{0\} \subset \widetilde{\Gamma}(m,l)$ and $0 \ast g(\tau)=g(\tau)$ for any $0 \leq \tau \leq 1$ and $g \in \Gamma(m,l)$. This immediately shows that 
		$$
		\widetilde{\gamma}(m,l) \leq \gamma(m,l).
		$$ 
		Let $\widetilde{g} \in \widetilde{\Gamma}(m,l)$ be such that $\widetilde{g}(\tau):=({h}_1(\tau), h_2(\tau)) \in S(m,l) \times \R$ for $0 \leq \tau \leq 1$, where $h_1(0)=u_{m,l}$, $h_1(1)={(u_{m,l})}_{s_0}$, and $h_2(0)=h_2(1)=0$. Define $g(\tau):=h_2(\tau) \ast h_1(\tau)$ for $0 \leq \tau \leq 1$, then $g \in \Gamma(m, l)$. And we have that $F(\widetilde{g}(\tau))=E(g(\tau))$. For any $ \epsilon>0$, we can find that there exists $\widetilde{g} \in \widetilde{\Gamma}(m,l)$ such that 
        $$
        \max_{0\leq \tau\leq 1}F(\widetilde{g}(\tau))={\widetilde{\gamma}(m,l)}+\epsilon.
        $$ 
        This in turn implies that 
		$$
		\gamma(m,l) \leq \max_{0\leq \tau \leq 1} E(g(\tau))=\max_{0\leq \tau \leq 1} F(\widetilde{g}(\tau)) = \widetilde{\gamma}(m,l)+\epsilon.
		$$ 
		Hence the proof is complete.
	\end{proof}
	
	\begin{lem} \label{mps2}
		Let $\eps>0$ and $\widetilde{g}_0 \in \widetilde{\Gamma}(m,l)$ be such that
		$$
		\max_{0 \leq \tau\leq 1} F(\widetilde{g}_0(\tau)) \leq \widetilde{\gamma}(m,l) +\eps.
		$$
		Then there exists $(u_0, s_0) \in S(m,l) \times \R$ such that
		\begin{itemize}
			\item[$(\textnormal{i})$] $\widetilde{\gamma}(m,l)-\eps \leq F(u_0, s_0) \leq \widetilde{\gamma}(m,l)+ \eps$;
			\item[$(\textnormal{ii})$] $\min_{0 \leq \tau \leq 1} \left\|(u_0, s_0)-\widetilde{g}_0(\tau)\right\|_{\mathcal{H}^1 \times \R} \leq \sqrt{\eps}$;
			\item[$(\textnormal{iii})$] $\left|\langle F'(u_0, s_0), z \rangle \right| \leq 2 \sqrt{\eps} \|z\|_{\mathcal{H}^1 \times \R}$ for any $z \in \widetilde{T}_{(u_0, s_0)}$, where
			$$
			\widetilde{T}_{(u_0, s_0)}:=\left\{(u, s) \in {\mathcal{H}^1 \times \R} : \langle u_0, u \rangle_2 =0, \langle L_z u_0, u \rangle_2=0 \right\}.
			$$
		\end{itemize}
	\end{lem}
	\begin{proof}
		Following closely the spirit of the proof of \cite[Lemma 2.3]{Jeanjean}, one can simply get the desired conclusion. We omit the proof for simplicity.
	\end{proof}
	
	\begin{lem} \label{ps}
		There exists a sequence $\{v_n\} \subset S(m,l)$ such that
		\begin{align} \label{pss}
			E(v_n)=\gamma(m,l)+o_n(1), \quad E'\mid_{S(m, l)}(v_n)=o_n(1), \quad P(v_n)=o_n(1).
		\end{align}
	\end{lem}
	\begin{proof}
		Observe first that, for any $n \in \N^+$, there exists $g_n \in \Gamma(m,l)$ such that
		$$
		\max_{0 \leq \tau \leq 1} E(g_n(\tau)) \leq \gamma(m,l) + \frac 1 n.
		$$
		Let $\widetilde{g}_n=(g_n, 0) \in \widetilde{\Gamma}(m,l)$. It then follows from Lemma \ref{mps1} that
		$$
		\max_{0 \leq \tau \leq 1} F(\widetilde{g}_n(\tau)) \leq \widetilde{\gamma}(m,l) + \frac 1 n.
		$$
		Therefore, by Lemma \ref{mps2}, there exists a sequence $\{(u_n, s_n)\} \subset S(m,l) \times \R$ such that
		\begin{itemize}
			\item[$(\textnormal{i})$] $\gamma(m,l)-\frac 1n \leq F(u_{n}, s_{n}) \leq \gamma(m,l)+ \frac 1 n$;
			\item[$(\textnormal{ii})$] $\min_{0 \leq \tau \leq 1} \left\|(u_n, s_n)-\widetilde{g}_n(\tau)\right\|_{\mathcal{H}^1 \times \R} \leq \sqrt{\frac 1n}$;
			\item[$(\textnormal{iii})$] $\left|\langle F'(u_n, s_n), z \rangle \right| \leq 2 \sqrt{\frac 1 n} \|z\|_{\mathcal{H}^1 \times \R}$ for any $z \in \widetilde{T}_{(u_n, s_n)}$, where
			$$
			\widetilde{T}_{(u_n, s_n)}:=\left\{(u, s) \in {\mathcal{H}^1 \times \R} : \langle u_n, u \rangle =0, \langle L_z u_n, u \rangle =0 \right\}.
			$$
		\end{itemize}
		Define $v_n:=s_n \ast u_n$. Let us show that $\{v_n\} \subset S(m,l)$ satisfies \eqref{pss}. It first follows from 
		$(\textnormal{i})$ that 
		$$
		E(v_n)=\gamma(m,l)+o_n(1).
		$$ 
		Let $z=(0, 0)$. It is clear that $z \in \widetilde{T}_{(u_n, s_n)}$. This readily implies from $(\textnormal{iii})$ that $P(v_n)=o_n(1)$. Next we are going to verify that $E'\mid_{S(m, l)}(v_n)=o_n(1)$. Define
		$$
		T_{v_n}S(m,l):=\left\{u \in \mathcal{H}^1 : \langle v_n, u \rangle_2 =0, \langle L_z v_n, u \rangle_2=0 \right\}.
		$$
		Let $\eta \in T_{v_n}S(m, l)$. Observe that
		\begin{align} \label{de}
			\langle E'(v_n), \eta \rangle =\langle F'(u_n, s_n), (\widetilde{\eta}_n, 0) \rangle,
		\end{align}
		where $\widetilde{\eta}_n:=(-s_n) \ast \eta$. Simple calculations lead to
		$$
		\int_{\R^d} u_n \overline{\widetilde{\eta}_n} \,dx=\int_{\R^d} v_n \overline{\eta} \,dx=0, \quad \int_{\R^d} L_z u_n \overline{\widetilde{\eta}_n} \,dx=\int_{\R^d} L_z v_n \overline{\eta} \,dx=0.
		$$
		This implies that $(\widetilde{\eta}_n, 0) \in \widetilde{T}_{(u_n, s_n)}$, because of $\eta \in T_{v_n}S(m, l)$. It follows from $(\textnormal{ii})$ that
		\begin{align} \label{sn}
			|s_n| \leq \min_{0 \leq \tau \leq 1} \|(u_n, s_n)-(g_n(\tau), 0)\|_{\mathcal{H}^1 \times \R} \leq \frac{1}{\sqrt{n}}.
		\end{align}
		Therefore, by \eqref{sn}, we have that, for any $n \in \N^+$ large enough,
		$$
		\|(\widetilde{\eta}_n, 0)\|_{\mathcal{H}^1 \times \R}^2 \lesssim \int_{\R^d} |\eta|^2 \,dx + e^{-2s_n} \int_{\R^d} |\nabla \eta|^2 \,dx + \int_{\R^d} V(e^{s_n} x) |\eta|^2 \,dx \lesssim \|\eta\|^2_{\mathcal{H}^1}.
		$$
		As a consequence, by \eqref{de} and $(\textnormal{iii})$, it holds that
		$$
		\left|\langle E'(v_n), \eta \rangle\right| \lesssim \frac{1}{\sqrt{n}} \|\eta\|_{\mathcal{H}^1}.
		$$
		This completes the proof.
	\end{proof}
	
	\begin{thm} \label{thm2}
		There exists $m_0>0$ such that, for any $0<m<m_0$, \eqref{equ} satisfying \eqref{c} has a solution $v_{m ,l} \in S(m,l)$ at the level $\gamma(m,l)$. 
	\end{thm}
	\begin{proof}
		It follows from Lemma \ref{ps} that there exists a sequence $\{v_n\} \subset S(m, l)$ be such that
		$$
		E(v_n)=\gamma(m, l)+o_n(1), \quad E'\mid_{S(m,l)}(v_n)=o_n(1), \quad P(v_n)=o_n(1).
		$$
		Observe that
		\begin{align*}
			\gamma(m, l)+o_n(1)&=E(v_n)-\frac{1}{\sigma d} P(v_n) \\
			&= \left(\frac 12 -\frac{1}{\sigma d}\right) \int_{\R^d} |\nabla v_n|^2 \,dx+\frac{1}{\sigma d} \int_{\R^d} \langle \nabla V(x), x \rangle |v_n|^2 \,dx + \int_{\R^d}V(x) |v_n|^2 \,dx.
		\end{align*}
		Thanks to $\sigma d>2$ and \eqref{v}, we then obtain that $\{v_n\}$ is bounded in $\mathcal{H}^1$. This implies that there exists $v_{m,l} \in \mathcal{H}^1$ such that $v_n \wto v_{m,l}$ in $\mathcal{H}^1$ as $n \to \infty$. {By Lemma \ref{cembedding}, we have that $v_n \to v_{m,l}$ in $L^{2+2\sigma}(\R^d)$ as $n \to \infty$ for any $0 \leq \sigma<\frac{2}{(d-2)^+}$}. Let $h \in \mathcal{H}^1$. If $G_1'(v_n)$ and $G_2'(v_n)$ are linearly independent, then we define
		\begin{align} \label{defetan}
		\eta_n:=h-\langle G_1'(v_n), h \rangle e_{1,n} -\langle G_2'(v_n), h \rangle e_{2,n},
		\end{align}
		where $\{e_{i,n}\} \subset \mathcal{H}^1$ is bounded and $\langle G_i'(v_n), e_{j, n}\rangle=\delta_{ij}$ for $i,j=1,2$. {It is straightforward to find that $\langle G_i'(v_n), \eta_n \rangle=0$ for $i=1,2$}.
		It then implies that $\{\eta_n\} \subset T_{v_n}S(m, l)$. {Observe that $\{\eta_n\} \subset \mathcal{H}^1$ is bounded by \eqref{defetan} and the fact that $\{v_n\} \subset \mathcal{H}^1$ is bounded. Since $E'\mid_{S(m,l)}(v_n)=o_n(1)$, then}
		$$
		\langle E'(v_n), \eta_n \rangle=o_n(1).
		$$
		Define 
		\begin{align} \label{deflm}
		\omega_n:=\langle E'(v_n), e_{1,n} \rangle, \quad \Omega_n:=\langle E'(v_n), e_{2,n} \rangle.
		\end{align}
		Therefore, we get that
		\begin{align} \label{equvn}
			-\frac 12 \Delta v_n + V(x) v_n-\omega_n v_n-\Omega_n L_z v_n=\lambda |v_n|^{2\sigma} v_n +o_n(1).
		\end{align}
		Since $\{v_n\} \subset \mathcal{H}^1$ and $\{e_{i,n}\} \subset \mathcal{H}^1$ are bounded, {then $\{\omega_n\}, \{\Omega_n\} \subset \R$ are bounded by \eqref{deflm}}. It then follows that there exist $\omega, \Omega \in \R$ such that $\omega_n \to \omega$ and $\Omega_n \to \Omega$ in $\R$ as $n \to \infty$. {Since $v_n \wto v_{m,l}$ in $\mathcal{H}^1$ as $n \to \infty$, by \eqref{equvn}, then $v_{m,l} \in \mathcal{H}^1$ solves the equation}
		\begin{align} \label{equv}
			-\frac 12 \Delta v_{m,l} + V(x) v_{m,l}-\omega v_{m,l}-\Omega L_z v_{m,l}=\lambda |v_{m,l}|^{2\sigma} v_{m,l}.
		\end{align}
		If $G_1'(v_n)$ and $G_2'(v_n)$ are linearly dependent, then
		$$
		T_{v_n}S(m,l)=\left\{u \in \mathcal{H}^1 : {\langle G_1'(v_n), u \rangle_2 }= 0 \right\}.
		$$
		In this case, we define
		\begin{align} \label{defetn1}
		\eta_n:=h-\langle G_1'(v_n), h \rangle e_{1,n},
		\end{align}
		where $\{e_{i,n}\} \subset \mathcal{H}^1$ is bounded and $\langle G_1'(v_n), e_{1, n}\rangle=1$.  {It follows from \eqref{defetn1} that $\langle G_1'(v_n), \eta_n \rangle= 0$. This infers that $\{\eta_n\} \subset T_{v_n}S(m,l)$. In addition, by \eqref{defetn1}, one has that $\{\eta_n\} \subset \mathcal{H}^1$ is bounded}. Define
		$$
		\omega_n:=\langle E'(v_n), e_{1,n} \rangle, \quad \Omega_n:=0.
		$$
		Thereby, we conclude that $\{v_n\} \subset \mathcal{H}^1$ satisfies \eqref{equvn} with $\Omega_n=0$. {Since $v_n \wto v_{m,l}$ in $\mathcal{H}^1$ as $n \to \infty$, then $v_{m,l} \in \mathcal{H}^1$ solves \eqref{equv} with $\Omega=0$}. {We now demonstrate that $v_{m,l} \in S(m,l)$. It is simple to obtain that $G_1(v_{m,l})=0$}, due to $v_n \to v_{m,l}$ in $L^2(\R^d)$ as $n \to \infty$. Moreover, using the similar spirit as the proof of Theorem \ref{thm1}, we are also able to show that $G_2(v_{m,l})=0$, i.e. $v_{m,l} \in S(m,l)$. {Testing \eqref{equvn} and \eqref{equv} by $v_n$ and $v_{m,l}$ respectively and using Lemma \ref{cembedding}, we then obtain that
		$$
		\int_{\R^d} \frac 12 |\nabla v_n|^2 +V(x) |v_n|^2 \,dx=\int_{\R^d} \frac 12 |\nabla v_{m,l}|^2  \,dx +V(x) |v_{m,l}|^2 \,dx +o_n(1).
		$$
		Since $E(v_n)=\gamma(m, l)+o_n(1)$, by applying again Lemma \ref{cembedding}, then we have that $E(v_{m,l})=\gamma(m,l)$.
		}
		This completes the proof.
	\end{proof}
	
	\begin{proof}[Proof of Theorem \ref{existence}] 
		Using Theorems \ref{thm1} and \ref{thm2} along with Lemma \ref{mps}, we then have the desired conclusion and the proof is complete.
	\end{proof}
	
	\section{Stability and Instability of Solutions} \label{dynamics}
	
	In this section, we are going to discuss orbital stability and instability of solutions to \eqref{equ} and \eqref{c} obtained in Theorem \ref{existence}. First we shall present orbital stability of the set of local minimizers. 
	
	
	\begin{proof}[Proof of Theorem \ref{stable}]
		Define
		$$
		\mathcal{B}(m,l):=\left\{ u\in S(m,l) : E(u)=e(m,l)\right\}.
		$$
		Suppose by contradiction that there exist $\eps_0>0$, $\{u_{0,n}\} \subset \mathcal{H}^1$, $u_0 \in \mathcal{B}(m,l)$ and $\{t_n\} \subset \R$ such that
		\begin{align} \label{stable1}
			\lim_{n \to \infty} \|u_{0,n}-u_0\|_{\mathcal{H}^1}=0,
		\end{align}
		\begin{align} \label{stable2}
			\inf_{u \in \mathcal{B}(m,l)} \|v_n(t_n, \cdot)-u\|_{\mathcal{H}^1} \geq \eps_0,
		\end{align}
		where $v_n \in C((-T_{max}, T_{max}), \mathcal{H}^1)$ is the unique solution to \eqref{equt} with initial data $u_{0,n} \in \mathcal{H}^1$. Moreover, we have that $\{v_{n}(t_n, \cdot)\} \subset B(r)$ for any $t_{n} \in (-T_{max}, T_{max})$. If not, by the continuity, then there exists $0<s_n <t_n$ such that $\{v_{n}(s_n, \cdot)\} \subset \partial B(r)$, because $u_{0,n} \subset B(r)$. With the help of the conservation laws in Lemma \ref{lwp}, Lemma \ref{mps0} and Theorem \ref{thm1}, we then derive that
		\begin{align*}
			E(u_{0,n})=E(v_{n}(s_n, \cdot)) \geq \inf_{u \in B(r) \backslash B(r/2)}E(u)>\inf_{u \in B(r/16)}E(u)\geq e(m,l).
		\end{align*}
		Then we reach a contradiction from the assumption. Therefore, we conclude that $T_{max}=+\infty$, i.e. the solutions exist globally in time. In view of \eqref{stable1}, we know that
		$$
		\lim_{n \to \infty}E(u_{0,n})=E(u_0), \quad \lim_{n \to \infty} M(u_{0,n})=M(u_0).
		$$
		Further, arguing as the proof of Theorem \ref{thm1}, we can get that
		$$
		\lim_{n \to \infty} L(u_{0,n})=L(u_0).
		$$
		Invoking the conservation laws in Lemma \ref{lwp} and the conservation of the angular momentum, we then find that
		$$
		\lim_{n \to \infty}E(v_{n}(t_n, \cdot))=E(u_0), \quad \lim_{n \to \infty} M(v_{n}(t_n, \cdot))=M(u_0), \quad \lim_{n \to \infty} L(v_{n}(t_n, \cdot))=L(u_0).
		$$
		In a similar way as the proof of Theorem \ref{thm1}, by Lemma \ref{cembedding}, we are able to show that $\{v_{n}(t_n, \cdot)\}$ is compact in $\mathcal{H}^1$,
		which clearly contradicts \eqref{stable2}. 
		It then follows that $\mathcal{B}(m,l)$ is orbitally stable and the desired conclusion follows. This completes the proof. 
	\end{proof}
	
	Next we shall investigate orbital instability of mountain type solutions to \eqref{equ} and \eqref{c}. For this, we first present an alternative variational characterization of the mountain pass energy level $\gamma(m,l)$, see for example \cite{BJ, Gou, S}. In what follows, we shall assume that the assumptions \eqref{v1} and \eqref{v2} hold.
	
	\begin{lem} \label{maxi}
		There exists $m_0>0$ such that, for any $0<m<m_0$ and $u \in S(m,l)$, the function $s \mapsto E(u_s)$ has exactly two critical points $0<s_{u, 1}<s_{u,2}$ on $(0, +\infty)$ such that $u_{s_{u, 1}} \in \mathcal{P}^+(m,l)$ and $u_{s_{u, 2}} \in \mathcal{P}^-(m,l)$, where 
		$$
		\mathcal{P}^+(m,l):=\left\{u \in S(m,l) : P(u)=0, Q(u)>0\right\}, 
		$$
		$$
		\mathcal{P}^-(m,l):=\left\{u \in S(m,l) : P(u)=0, Q(u)<0\right\},
		$$ 
		$$
		Q(u):=\int_{\R^d} |\nabla u|^2+2\langle \nabla V(x), x \rangle +\sum_{i,j=1}^d\partial_{ij}V(x)|u|^2 -\frac{\lambda \sigma d (\sigma d-1)}{\sigma+1} |u|^{2 \sigma +2}\,dx{=\frac{d^2}{ds^2} E(u_s) \mid_{s=1}}.
		$$
		Moreover, the function $s \mapsto E(u_s)$ is concave on $[s_{u, 2}, + \infty)$.
	\end{lem}
	\begin{proof}
		Observe first that
		$$
		E(u_s)=\int_{\R^d} \frac {s^2}{2} |\nabla u|^2 + V\left(\frac{x}{s}\right) |u|^2 - \frac{\lambda s^{\sigma d}}{\sigma+1} |u|^{2 \sigma +2} \,dx.
		$$
		It is straightforward to compute that
		\begin{align*}
			\frac{d}{ds} E(u_s)=\int_{\R^d} s |\nabla u|^2 - \frac{1}{s^2}\langle \nabla V\left(\frac{x}{s}\right), x \rangle  |u|^2 - \frac{\lambda \sigma d s^{\sigma d-1}}{\sigma+1} |u|^{2 \sigma +2} \,dx=\frac 1 s P(u_s),
		\end{align*}
		\begin{align*}
			\frac{d^2}{ds^2} E(u_s)&=\int_{\R^d} |\nabla u|^2+\frac{1}{s^3}\left(2\langle \nabla V\left(\frac{x}{s}\right), x \rangle +\frac{1}{s} \sum_{i,j=1}^d\partial_{ij}V\left(\frac{x}{s}\right)x_ix_j \right)|u|^2 \\
			& \quad -\frac{\lambda \sigma d (\sigma d-1)s^{\sigma d-2}}{\sigma+1} |u|^{2 \sigma +2}\,dx.
		\end{align*}
		It then follows from \eqref{v2} that the function $s \mapsto \frac{d^2}{ds^2} E(u_s)$ is nonincreasing on $(0, +\infty)$, because of $\sigma d>2$. In addition, from \eqref{v2}, it holds that 
		$$
		\lim_{s \to 0}\frac{d^2}{ds^2} E(u_s)=+\infty, \quad \lim_{s \to + \infty}\frac{d^2}{ds^2} E(u_s)=-\infty.
		$$
		This readily implies the function $s \mapsto \frac{d^2}{ds^2} E(u_s)$ has exactly one zero in $(0, +\infty)$. Moreover, we are able to conclude from \eqref{v} and \eqref{v1} that
		$$
		\lim_{s \to 0}\frac{d}{ds} E(u_s)=-\infty, \quad \lim_{s \to + \infty}\frac{d}{ds} E(u_s)=-\infty.
		$$
		Observe that $c_0 (1+|x|^2)^{k/2} \leq V(x) \leq C_0 (1+|x|^2)^{k/2}$ for $|x|>R$ by Assumption \ref{a1}. Consequently, we find that 
		$$
		\lim_{s \to 0}E(u_s)=+\infty, \quad \lim_{s \to + \infty}E(u_s)=-\infty.
		$$
		If 
		$$
		\max_{s>0} \frac{d}{ds} E(u_s) \leq 0,
		$$ 
		then the function $s \mapsto E(u_s)$ is nonincreasing on $(0, +\infty)$. However, applying \eqref{below}, we know that, for any $r>0$, there exists $m_0>0$ such that \eqref{below1} and \eqref{below2} hold.
		This then implies that 
		$$
		\max_{s>0} \frac{d}{ds} E(u_s) > 0.
		$$ 
		Therefore, the function $ s \mapsto \frac{d}{ds} E(u_s) $ has exactly two zeros on $(0, +\infty)$. At this point, we can conclude the proof.
	\end{proof}
	
	\begin{lem} \label{vc}
		There exists $m_0>0$ such that, for any $0<m<m_0$, it holds
		$$
		\gamma(m,l)=\inf_{u \in \mathcal{P}^-(m,l)} E(u).
		$$
	\end{lem}
	\begin{proof}
		Let $\widetilde{g} \in \widetilde{\Gamma}(m,l)$. Denote by $\widetilde{g}(\tau):=(h_1(\tau), h_2(\tau))$ for $0 \leq \tau \leq 1$ with $h_2(0)=h_2(1)=0$. {Since $u_{m,l} \in S(m,l)$ is a solution to \eqref{equ} and \eqref{c}, then $P(u_{m,l})=0$. Moreover, in the spirit of the proof of Lemma \ref{maxi}, we know that
		$$
	       Q(u_{m,l})=\frac{d^2}{ds^2} E((u_{m,l})_s) \mid_{s=1}>0.
		$$
		It then follow that $u_{m,l} \in \mathcal{P}^+(m,l)$.} Since $\widetilde{g}(0)=(u_{m,l}, 0)$, by Lemma \ref{maxi}, then $s_{h_2(0) \ast h_1(0), 1}=s_{h_1(0), 1}=1$ and $s_{h_2(0) \ast h_1(0), 2}=s_{h_1(0), 2}>1$. Since $\widetilde{g}(1)=((u_{m,l})_{s_0}, 0)$ and $E((u_{m,l})_{s_0})<0$, by Lemma \ref{maxi}, then $s_{h_2(1) \ast h_1(1), 2}=s_{h_1(1), 2}<1$. As a consequence, there exists $0<\widetilde{\tau}<1$ such that $s_{h_2(\widetilde{\tau}) \ast h_1(\widetilde{\tau}), 2}=1$. Therefore, it holds that
		\begin{align*} 
			\max_{0 \leq \tau \leq 1} F(\widetilde{g}(\tau)) \geq E(h_2(\widetilde{\tau}) \ast h_1(\widetilde{\tau})) \geq \inf_{u \in \mathcal{P}^-(m,l)} E(u).
		\end{align*}
		This implies that 
		\begin{align} \label{vc1}
			\widetilde{\gamma}(m,l) \geq \inf_{u \in \mathcal{P}^-(m,l)} E(u).
		\end{align}
		Let $ u \in \mathcal{P}^{-}(m,l)$. Invoking Lemma \ref{maxi}, we know that there exists $s_{u, 1}>0$ such that $u_{s_{u,1}} \in \mathcal{P}^+(m,l)$. Define a path $\widetilde{g}$ by
		$$
		\widetilde{g}(\tau):=(((1-\tau)s_{u,1} + \tau s_0))^{\frac d2 }u(((1-\tau)s_{u,1} + \tau s_0)x), 0), \quad 0 \leq \tau \leq 1,
		$$
		where $s_0>0$ such that $E(u_{s_0})<0$. It then follows that
		$$
		E(u)=F(u, 0)=\max_{0 \leq \tau \leq 1} F(\widetilde{g}(\tau)) \geq \widetilde{\gamma}(m,l),
		$$
		which implies that
		\begin{align*} 
			\inf_{u \in \mathcal{P}^-(m,l)} E(u) \geq \widetilde{\gamma}(m,l).
		\end{align*}
		This along with \eqref{vc1} and Lemma \ref{mps1} lead to the desired conclusion. Therefore, the proof is complete.
	\end{proof}
	The following lemma has been proved in \cite{AMS}. Here we give a explicit proof for reader's convenience.  
	\begin{lem} \label{virial}
		Let $u \in C([0, T_{max}), \mathcal{H}^1)$ be the solution to \eqref{equt} with initial data $u_0 \in \mathcal{H}^1$. Define 
		$$
		\mathcal{M}[u(t)]:=\int_{\R^d} |x|^2 |u(t, x)|^2 \,dx.
		$$
		Then it holds that
		\begin{align*}
			\frac{d^2}{dt^2}\mathcal{M}[u(t)]=2\int_{\R^d} |\nabla u|^2 \,dx-2 \int_{\R^d} \langle \nabla V(x), x \rangle |u|^2 \,dx-\frac{2d \sigma \lambda}{\sigma +1} \int_{\R^d} |u|^{2 \sigma +2} \,dx.
		\end{align*}
	\end{lem}
	\begin{proof} 
	{To establish the desired conclusion, we shall follow closely the regularization arguments developed in the proof of \cite[Proposition 6.5.1]{Caz}. Let us start with computing the first derivative of 
$\mathcal{M}[u(t)]$. Define
$$
\mathcal{M}_{\eps}[u(t)]:=\int_{\R^d} e^{-2\eps |x|^2} |x|^2 |u(t, x)|^2 \,dx, \quad \eps>0.
$$	
It is simple to find that
$$
\frac{d}{dt}\mathcal{M}_{\eps}[u(t)]=2 \mbox{Im} \int_{\R^d} e^{-2\eps |x|^2}\left(1-2\eps|x|^2\right)\left(\nabla u \cdot x\right) \overline{u} \,dx.
$$
Making use of the elements presented in the proof of \cite[Proposition 6.5.2]{Caz} and taking the limit as $\eps \to 0$, we then have that
$$
\frac{d}{dt}\mathcal{M}[u(t)]=2 \mbox{Im} \int_{\R^d}  \left(\nabla u \cdot x\right) \overline{u} \,dx.
$$
Next we shall calculate the second derivative of $\mathcal{M}[u(t)]$. Let us first assume that $u_0 \in \mathcal{H}^2$ and $u \in C([0, T_{max}), \mathcal{H}^2) \cap C^1((0, T_{max}), L^2(\R^d)) $ is the solution to \eqref{equt} with the initial data $u_0$, where the space $\mathcal{H}^2$ is defined by the completion of $C_0^{\infty}(\R^d)$ under the norm
$$
\|u\|_{\mathcal{H}^2}:=\|u\|_{\mathcal{H}^1} +\|\Delta u\|_2.
$$
Define
$$
\widetilde{\mathcal{M}}_{\eps}(t):=2\mbox{Im} \int_{\R^d} e^{-\eps |x|^2}\left(\nabla u \cdot x\right) \overline{u} \,dx.
$$
Proceeding as Step 1 in the proof of \cite[Proposition 6.5.1]{Caz} and using the density, we may assume without restrcition that $u \in C^1((0, T_{max}), \mathcal{H}(\R^d))$. It then follows that

\begin{align*}
			\frac{d}{dt}\widetilde{\mathcal{M}}_{\eps}(t)=2 \mbox{Im} \int_{\R^d} e^{-\eps |x|^2} \left(\nabla u_t \cdot x\right) \overline{u} \,dx+2 \mbox{Im} \int_{\R^d} e^{-\eps |x|^2} \left(\nabla u \cdot x\right) \overline{u}_t \,dx.
		\end{align*}
		Observe that
		$$
	e^{-\eps |x|^2} \left(\nabla u_t \cdot x\right) \overline{u} =\nabla \cdot (x e^{-\eps |x|^2}  u_t  \overline{u})-de^{-\eps |x|^2}  u_t  \overline{u}- e^{-\eps |x|^2} \left(\nabla \overline{u} \cdot x\right) u_t+2 \eps |x|^2 e^{-\eps |x|^2} u_t  \overline{u}.
		$$
Consequently, it holds that
		$$
	\frac{d}{dt}\widetilde{\mathcal{M}}_{\eps}(t)=2 \mbox{Re} \int_{\R^d} \left(H u-\Omega L_z u-\lambda |u|^{2\sigma} u\right) \left(2 \left(\nabla \overline{u} \cdot x\right) +d \overline{u} - 2 \eps |x|^2 \overline{u}\right) e^{-\eps |x|^2} \,dx.
		$$
	At this point, invoking the divergence theorem and adapting the ingredients presented in Step 1 of the proof of \cite[Proposition 6.5.1]{Caz}, we are able to show that
	\begin{align*}
	\frac{d}{dt}\widetilde{\mathcal{M}}_{\eps}(t)=2\int_{\R^d} |\nabla u|^2 \,dx-2 \int_{\R^d} \langle \nabla V(x), x \rangle |u|^2 \,dx-\frac{2d \sigma \lambda}{\sigma +1} \int_{\R^d} |u|^{2 \sigma +2} \,dx+o_{\eps}(1).
	\end{align*}
		Taking the limit as $\eps \to 0$, we now derive that
		$$
	\frac{d^2}{dt^2}\mathcal{M}[u(t)]=2\int_{\R^d} |\nabla u|^2 \,dx-2 \int_{\R^d} \langle \nabla V(x), x \rangle |u|^2 \,dx-\frac{2d \sigma \lambda}{\sigma +1} \int_{\R^d} |u|^{2 \sigma +2} \,dx.
		$$
		If $u_0 \in \mathcal{H}^1$, then we know that there exists $\{(u_0)_n\} \subset \mathcal{H}^2$ such that $(u_0)_n \to u_0$ in $\mathcal{H}^1$ as $n \to \infty$. Let $\{u_n\} \subset C([0, T_{max}), \mathcal{H}^2) \cap C^1((0, T_{max}), L^2(\R^d))$ be the solutions to \eqref{equt} to the initial data $\{(u_0)_n\}$. Define
		$$
	       U_n(t):=2\int_{\R^d} |\nabla u_n|^2 \,dx-2 \int_{\R^d} \langle \nabla V(x), x \rangle |u_n|^2 \,dx-\frac{2d \sigma \lambda}{\sigma +1} \int_{\R^d} |u_n|^{2 \sigma +2} \,dx.
		$$
		We then conclude that
		\begin{align}  \label{v111}
	       \int_{\R^d} |x|^2 |u_n|^2 \,dx= \int_{\R^d} |x|^2 |(u_0)_n|^2 \,dx + 4 t \mbox{Im} \int_{\R^d}  \left(\nabla {u_n} \cdot x\right) \overline{u_n} \,dx + 2\int_0^t \int_0^s U_n(s) \,dsdt.
		\end{align}
		In addition, arguing as the proof of \cite[Corollary 6.5.3]{Caz}, we are able to derive that $x u_n \to x u$ in $ C([0, T_{max}), L^2(\R^d))$ as $n \to \infty$.
	Therefore, using the continuous dependence and taking the limit as $n \to \infty$ in \eqref{v111}, we obtain the desired conclusion. This completes the proof.
 			}
	\end{proof}
	
	\begin{proof}[Proof of Theorem \ref{unstable}] Let $v_{m,l} \in S(m,l)$ be the solution to \eqref{equ} satisfying \eqref{c} at the level $\gamma(m,l)$ obtained in Theorem \ref{existence}. We are going to show that $v_{m,l}$ is strongly unstable. Note that $(v_{m,l})_s \to v_{m,l}$ as $s \to 1^-$ in $\mathcal{H}^1$. Let $s>1$ be such that $s$ is enough close to $1$. Define $u_0=(v_{m,l})_s$. Then, by Lemma \ref{maxi}, it holds that $E(u_0)<\gamma(m,l)$ and $P(u_0)<0$. And there exists $0<s_{u_0, 2}<1$ such that $(u_0)_{s_{u_0, 2}} \in \mathcal{P}^-(m,l)$. Let $u \in C([0, T_{max}), \mathcal{H}^1)$ be the solution to \eqref{equt} with initial datum $u_0$. We shall prove that $u(t)$ blows up in finite time. Since the function $s \to E((u_0)_s)$ is concave on $[s_{u_0,2}, 1]$, then
		$$
		E(u_0)-E((u_0)_{s_{u_0, 2}})=\frac{d}{ds} E((u_0)_s) \mid_{s=\xi} (1-s_{u_0, 2}) \geq \frac{d}{ds} E((u_0)_s) \mid_{s=1} (1-s_{u_0,2})=P(u_0)(1-s_{u_0, 2}), 
		$$
		where $s_{u_0, 2} \leq \xi \leq 1$. This gives that
		\begin{align} \label{bbelow}
			E(u_0) \geq \gamma(m,l)+P(u_0)(1-{s_{u_0 ,2}})> \gamma(m,l)+P(u_0).
		\end{align}
		Next we are going to conclude that $P(u(t))<-\delta$ for any $0 \leq t<T_{max}$, where $\delta:=\gamma(m,l)-E(u_0)>0$. We now assume by contradiction that there exists $0<t_1<T_{max}$ such that $s_{u(t_1), 2}=1$, $0<s_{u(t), 2}<1$ for $0<t<t_1$ and $P(u(t_1))=0$. Reasoning as the proof of \eqref{bbelow} and using the conservation of the energy, we then have that
		$$
		P(u(t)) \leq E(u(t))-\gamma(m,l)=E(u_0)-\gamma(m,l)=-\delta<0, \quad 0<t<t_1.
		$$
		Hence it holds that $P(u(t_1)) \leq -\delta<0$. This is impossible. It then implies that $0<s_{u(t), 2}<1$ and $P(u(t))<0$ for any $0 \leq t<T_{max}$. Thus, arguing as the proof of \eqref{bbelow}, we obtain that $P(u(t))<-\delta$ for any $0\leq t<T_{max}$. At this point, using Lemma \ref{virial}, we get that
		$$
		\frac{d^2}{dt^2}\mathcal{M}[u(t)]=2 P(u(t))<-2 \delta, \quad 0 \leq t <T_{max},
		$$
		which readily infers that the solution $u$ blows up in finite time. This completes the proof. 
	\end{proof}

	\section{Regularity of Minimizers} \label{regularity}
	
	In this section, as a complement to the study carried out in \cite{NSS}, we shall further investigate regularity of global minimizers to \eqref{gmin1}. As a direct consequence of the discussion, we obtain the same result for the local minimizers to \eqref{lmin1}. We say that $u \in S(m,l)$ is \emph{regular} if $G_1'(u)$ and $G_2'(u)$ are linearly independent. The main result of this section reads as follows.
	
	
	\begin{prop} \label{thm4}
		Let $u \in S(m,l)$ be a minimizer to \eqref{gmin1}. If $\frac {l}{m} \not\in \mathbb{\Z}$, then $u$ is regular. In addition, regularity also holds if $\frac{l}{m}=n\in \Z$ and if $u$ is not an eigenfunction of $L_z$ at eigenvalue $n$. On the other hand, if $u$ is the eigenfunction of $L_z$ at some eigenvalue $n \in \Z$, then $u$ solves \eqref{equ} with $\Omega=0$, i.e.
		\begin{align} \label{equu}
			-\frac 12 \Delta u + V(x) u-\omega u=\lambda |u|^{2\sigma} u.
		\end{align}
	\end{prop}
	\begin{proof}
		Let $\lambda_1, \lambda_2 \in \R$ be such that 
		$$
		\lambda_1 G_1'(u)+\lambda_2 G_2'(u)=0\,\,\,\Longleftrightarrow\,\,\, \lambda_1 u + \lambda_2 L_z u=0.
		$$
		Since $u \in S(m,l)$, then $\lambda_1 m + \lambda_2l=0$. If $\lambda_2=0$, then $\lambda_1=0$. This implies that $u$ is regular and the proof is complete. Now we assume that $\lambda_2 \neq 0$. This shows that
		$$
		L_z u=-\frac{\lambda_1}{\lambda_2} u=\frac{l}{m} u.
		$$
		Observe that the operator $L_z$ is essentially self-adjoint on $\mathcal{H}^1$ with purely discrete spectrum $\sigma(L_z)=\mathbb{Z}$. This immediately implies that if $\frac{l}{m} \not\in \mathbb{Z}$, then $u$ is regular. Moreover, if $\frac{l}{m} \in \mathbb{Z}$ and $u$ is not an eigenfunction of $L_z$, then $u$ is regular as well. Let us now suppose that $u$ is an eigenfunction of $L_z$ corresponding to some eigenvalue $n \in \mathbb{N}$. This means that $u$ satisfies the eigenvalue problem $L_z u=n u$ with $n=\frac{l}{m}$. We denote by $\mathcal{L}_n$ the associated eigenspace defined by
		$$
		\mathcal{L}_n:=\left\{u \in \mathcal{H}^1 : L_z u=n u \right\}.
		$$
		Since $u \in S(m,l) \cap \mathcal{L}_n$ is a minimizer to \eqref{gmin1}, then it is a minimizer to the minimization problem
		\begin{align} \label{gmin11}
			e_n(m):=\inf_{u \in S(m) \cap \mathcal{L}_n} E(u),
		\end{align}
		where
		$$
		S(m):=\left\{u \in \mathcal{H}^1 : M(u)=m\right\}.
		$$
		Next we are going to demonstrate that $u$ satisfies \eqref{equu}. For simplicity, we shall only deal with the case in $\R^3$. In cylindrical coordinates $(r, \varphi, z) \in \R^3$, we have that $L_z=-\textnormal{i} \partial_{\varphi}$. Here we denote by $L^2_{cyl}$ the weighted $L^2$-space defined by $L^2_{cyl}:=L^2((0, +\infty) \times \R, r drdz)$ and denote by $\mathcal{H}^1_{cyl}$ the corresponding Sobolev space defined by
		$$
		\mathcal{H}^1_{cyl}:=\left\{f  \in L^2_{cyl} : \nabla_{r, z} f, \, \frac f r, \, V^{1/2} f \in L^2_{cyl}\right\},
		$$
		where $\nabla_{r, z}:=(\partial_r, \partial_z)$ and $V=V(x)$ is identified with $V(r, z)$ in cylindrical coordinates. Since $u \in \mathcal{L}_n$, then there exists $f \in \mathcal{H}^1_{cyl}$ such that 
		\begin{align} \label{uf}
			u(r,\varphi,z)=f(r, z) e^{\textnormal{i} n\varphi}.
		\end{align}
		It then follows that
		$$
		\int_{\R^3} |u|^2 \,dx=2\pi \int_{\R} \int_0^{\infty} |f|^2 r \,drdz,
		$$
		$$
		\int_{\R^3} |\nabla u|^2 \,dx =2 \pi \int_{\R}\int_0^{\infty}\left(|\nabla_{r,z} f|^2 + \frac{n^2}{r^2} |f|^2\right) r \,drdz,
		$$
		where
		$$
		\left(\begin{matrix}
			\partial_x u \\ \partial_y u
		\end{matrix}\right)=e^{\textnormal{i}n\varphi}
		\left(\begin{matrix}
			\cos\varphi & -\sin\varphi\\
			\sin\varphi & \cos\varphi
		\end{matrix}\right) \,
		\left(\begin{matrix}
			\partial_r f \\ \tfrac{\textnormal{i}n}{r} f
		\end{matrix}\right).
		$$
		Since $u \in S(m) \cap \mathcal{L}_n$ is a minimizer to \eqref{gmin11}, then $f \in S_{cyl}(m)$ given by \eqref{uf} is a minimizer to the minimization problem
		$$
		e_n(m)=\inf_{f \in S_{cyl}(m)} E_n(f),
		$$
		where
		$$
		E_n(f)=2\pi\int_{\R}\int_0^\infty \left(\frac12 |\nabla_{r,z} f|^2+\frac{n^2}{2r^2} |f|^2+V |f|^2
		-\frac{\lambda}{\sigma+1} |f|^{2\sigma+2}\right)\,rdrdz,
		$$
		$$
		S_{cyl}(m):=\left\{f \in \mathcal{H}^1_{cyl} : M_{cyl}(f)=m\right\}, \quad M_\text{cyl}(f):=2\pi\int_{\R}\int_0^\infty |f|^2\,rdrdz.
		$$
		We infer that there exists $\omega \in \R$ such that $f \in S_{cyl}(m)$ satisfies the equation
		\begin{align} \label{eqf}
			-\frac{1}{2}\nabla_{r,z}^2f+\frac{n^2}{2r^2}f+Vf-\omega f=\lambda|f|^{2\sigma}f
		\end{align}
		weakly in $\mathcal H^1_\text{cyl}$. Multiplying \eqref{eqf} by $e^{\textnormal{i}n \varphi}$ and converting back to Cartesian coordinates then shows that $u \in S(m) \cap \mathcal{L}_n$ solves the equation
		\begin{align} \label{eq}
			-\frac{1}{2}\Delta u+V u-\omega u=\lambda|u|^{2\sigma}u
		\end{align}
		weakly in $\mathcal L_n$. Finally, noting that the orthogonality of $\mathcal L_k$ and $\mathcal L_n$ for $k\not=n$, we conclude that \eqref{eq} holds weakly in $\mathcal {H}^1$. This completes the proof.
	\end{proof}
	
	\begin{rem}
		It follows from Proposition \ref{thm4} and Lagrange’s theorem that if $u \in S(m,l)$ is a regular minimizer to \eqref{gmin1}, then there exist $\omega, \Omega \in \R$ such that $u$ satisfies \eqref{equ}. However, the existence of a minimizing profile $f \in S_{cyl}(m)$ only implies the corresponding $u\in \mathcal{L}_{n}$ being the minimizer of the reduced problem \eqref{gmin11}, not necessarily \eqref{gmin1}. Whether or not the doubly constrained problem \eqref{gmin1} admits non-regular minimizer $u\in \mathcal{L}_n$ remains an open problem.
	\end{rem}
	
	As a straightforward consequence of the arguments adapted to prove Proposition \ref{thm4}, we are also able to conclude the regularity of local minimizers to \eqref{lmin1}. 
	
	\begin{cor} \label{thm5}
		Let $u \in B(r)$ be a minimizer to \eqref{lmin1}, then Proposition \ref{thm4} remains true.
	\end{cor}
	
      {
      \begin{rem}
       We are only able to show the regularity of minimizers to \eqref{lmin1}. However, the regularity of mountain pass type solutions is unknown.
     \end{rem}
      }

	\section*{Conflict of Interest Statement}
	
	The authors declare that there is no conflict of interest.
	
	\section*{Data Availability Statement}
	
	We affirm that our paper has no associated data.

\end{document}